\documentclass[10pt]{amsart}
\usepackage{hyperref,amssymb}
\usepackage{bbm}

\usepackage[margin=1in]{geometry}
\usepackage{amsmath,amssymb,amsthm}
\usepackage{enumitem}

\newtheorem{theorem}{Theorem}[section]
\newtheorem{conjecture}[theorem]{Conjecture}
\newtheorem{lemma}[theorem]{Lemma}
\newtheorem{definition}[theorem]{Definition}
\newtheorem*{remark}{Remark}

\newcommand{\cost}{\operatorname{cost}}
\newcommand{\calA}{\mathcal A}
\newcommand{\calQ}{\mathcal Q}

\title{Proof of Thomassen's Conjecture on Highly connected subgraphs with large chromatic number}
\author{Achintya Raya Polavarapu}
\email{apolavarapu6@gatech.edu}

\address{School of Mathematics, Georgia Institute of Technology, Atlanta GA 30332.}

\date{\today}

\subjclass[2020]{Primary 05C15; Secondary 05C35, 05C40, 05C70}
\begin{document}

\begin{abstract}
For integers $k\ge 1$ and $m\ge 2$, let $g(k,m)$ be the least integer $n\ge 1$ such that every graph with chromatic number at least $n$ contains a $(k+1)$-connected subgraph with chromatic number at least $m$.  We prove that
\[
 g(k,m)\le \max(m+2k-2,\,3k+1)
\]
for all $k\ge 1$ and $m\ge 2$, establishing the 1983 conjecture of Thomassen~\cite{Thomassen} that $g(k,k+1)\le 3k+1$. The key new ingredient is a Hall-feasibility argument replacing the final numerical step in the proof of Nguyen~\cite{Nguyen}.
\end{abstract}

\maketitle

\section{Introduction}
\noindent For integers $k\ge 1$ and $m\ge 2$, define $g(k,m)$ to be the least integer $n\ge 1$ such that every graph with chromatic number at least $n$ contains a $(k+1)$-connected subgraph with chromatic number at least $m$. Thus $g(k,2)$ is simply the least $n$ such that every graph with chromatic number at least $n$ contains a $(k+1)$-connected subgraph, $g(1,m)=\max(m,3)$ for all $m\ge 2$, and the conjecture of Thomassen~\cite{Thomassen} asserts that $g(k,k+1)\le 3k+1$ for all $k\ge 1$.

\begin{conjecture}[Thomassen~\cite{Thomassen}, 1983]
\label{conj:thomassen}
For every integer $k\ge 1$, every graph with chromatic number more than $3k$ contains a $(k+1)$-connected subgraph with chromatic number more than $k$.
\end{conjecture}

\noindent The conjecture originated as a theorem with an incorrect proof. The study of $g(k,m)$ was initiated by Alon, Kleitman, Saks, Seymour, and Thomassen~\cite{AKSST}, who sought a correct substitute and established
\[
 \max(m+k-1,2k+1)\le g(k,m)\le \max(m+10k^2,100k^3+1).
\]
The upper bound was subsequently improved by Chudnovsky, Penev, Scott, and Trotignon~\cite{CPST} to
\[
 g(k,m)\le \max(m+2k^2,2k^2+k+1),
\]
and by Penev, Thomass\'e, and Trotignon~\cite{PTT} to
\[
 g(k,m)\le \max(m+2k-2,2k^2+1).
\]
Most recently, Gir\~ao and Narayanan~\cite{GN} proved that $g(k-1,k)\le 7k+1$, and Nguyen~\cite{Nguyen} refined their argument to
obtain
\[
 g(k,m)\le \max\!\left(m+2k-2,\left\lceil\!\left(3+\tfrac{1}{16}\right)k\right\rceil\right),
\]
making a step closer to Conjecture~\ref{conj:thomassen}. In this paper we prove the exact bound, thereby establishing Conjecture~\ref{conj:thomassen}.

\begin{theorem}
\label{thm:main}
For all integers $k\ge 1$ and $m\ge 2$,
\[
 g(k,m)\le \max(m+2k-2,3k+1).
\]
In particular, $g(k,k+1)\le 3k+1$ for all $k\ge 1$.
\end{theorem}

\subsection*{Proof strategy}
Our proof follows the framework of Nguyen~\cite{Nguyen}, which works with minimal $C$-inextensible graphs equipped with good templates. After a sequence of reductions, one is left with a family of stable residual pieces $Y_i$, each carrying an available-color list $L_i$, and the proof requires assigning each piece a distinct color from its list.  Nguyen handles this by establishing a uniform lower bound $|L_i|\ge|I|$ for all $i$, which suffices for a greedy assignment but forces the palette size to exceed $3k$ by a factor of $\frac{1}{16}$.

\medskip\noindent Our key observation is that greedy assignment requires more than necessary.  By Hall's theorem, a system of distinct representatives exists if and only if every subfamily $\calA$ satisfies $|\bigcup_{i\in\calA}L_i|\ge|\calA|$.  We prove this Hall condition directly by showing that no Hall obstruction can exist at any size, via a case analysis on $|\calA|$.  Large obstructions are ruled out by a structural result already present in Nguyen's proof; medium-sized obstructions are ruled out by the global inequality $\sum x_i < kt'$ inherited from the cost bound on the template; and small obstructions are ruled out by a direct list-size estimate that requires only $|C|\ge 3k$. Working with the exact condition for distinct representatives rather than a sufficient one is what closes the $\frac{1}{16}$ gap.

\subsection*{Organization}

Section~2 recalls the template framework and the minimal inextensibility machinery of Nguyen.  Section~3 introduces admissible light decompositions and isolates the recoloring budget that motivates the Hall analysis. Section~4 carries out the Hall-feasibility argument and derives Theorem~\ref{thm:main}.

\section{Templates and minimal obstructions}

\noindent We begin by recalling the template framework (see \cite{PTT, GN} for more information) that converts the global statement about highly chromatic graphs into a local obstruction problem. Following Nguyen, the right objects are minimal $C$-inextensible graphs equipped with good templates: once one passes to such a minimal obstruction, the coloring constraints are encoded by a bounded-cost template, while the structural conclusion of Nguyen's connectivity theorem remains available. This reduces the proof of the exact Nguyen--Thomassen bound to a chromatic lower bound for minimal $C$-inextensible graphs.

\medskip\noindent Let $G$ be a graph and let $C$ be a finite set of colors.

\begin{definition}[Template]
A $C$-template is a triple
\[
 T=(S,c,F),
\]
where $S\subseteq V(G)$, $c:S\to C$ is a proper coloring of $G[S]$, and for every $v\in V(G)\setminus S$ there is a forbidden set $F(v)\subseteq C$.
\end{definition}

\noindent A $C$-coloring $\varphi$ of $G$ \emph{respects} $T$ if
\[
 \varphi|_S=c
 \qquad\text{and}\qquad
 \varphi(v)\notin F(v)\quad\text{for every }v\notin S.
\]

\noindent For $X\subseteq V(G)\setminus S$, define its forbidden weight by
\[
 w(X)=\sum_{v\in X}|F(v)|.
\]

\noindent The cost of a template is
\[
 \cost_k(T)=k|S|+\sum_{v\in V(G)\setminus S}|F(v)|.
\]

\begin{definition}[$C$-inextensible]
A graph $G$ is $C$-inextensible if there exists a $C$-template $T=(S,c,F)$ such that
\[
 \cost_k(T)<2k^2,
 \qquad
 |F(v)|\le k\quad(v\notin S),
\]
and no proper $C$-coloring of $G$ respects $T$.
\end{definition}

\begin{lemma}[Empty template]
If $|C|=\chi(G)-1$, then $G$ is $C$-inextensible.
\end{lemma}

\begin{proof}
Take $S=\varnothing$ and $F(v)=\varnothing$ for every $v\in V(G)$.  A respecting coloring is simply a proper coloring of $G$ using colors from $C$. This would imply $\chi(G)\le |C|=\chi(G)-1$, impossible.
\end{proof}

We use the following consequence of Nguyen's template framework.

\begin{lemma}[Good template principle, Nguyen~\cite{Nguyen}]
\label{lem:good-template}
Let $G$ be $C$-inextensible.  Among all witnessing templates, choose one with $|S|$ maximal.  Then every uncolored vertex satisfies
\[
 |F(v)|\le k-1.
\]
\end{lemma}

A template satisfying the conclusion of Lemma~\ref{lem:good-template} will be called good. We also use Nguyen's connectivity theorem for minimal inextensible graphs.

\begin{theorem}[Nguyen~\cite{Nguyen}]
\label{thm:nguyen-connectivity}
Let $H$ be an induced subgraph minimal with respect to being $C$-inextensible.  Then $H$ is $(k+1)$-connected.  Moreover,
\[
 |H|>|C|-k+1.
\]
\end{theorem}

\section{Reduction to an exact decomposition criterion}

\noindent Before turning to the final Hall argument, it is useful to isolate the abstract recoloring mechanism that underlies the proof.  If, after partially coloring the graph, the remaining vertices can be partitioned into light stable pieces whose total number is within the budget $2k-|S|+1$, then each piece has a large available color list and the template can be completed. Although the final proof in Section~4 works directly with Nguyen's residual reduction rather than by invoking this criterion formally, the present section explains why the conclusion should be viewed as a controlled list-coloring problem on stable pieces.

\begin{definition}[Admissible light decomposition]
Let $T=(S,c,F)$ be a good $C$-template on $H$.  An admissible light decomposition consists of:
\begin{enumerate}[label=(\roman*)]
\item a set $U\subseteq V(H)\setminus S$ colored by a proper coloring using $q$ colors from $C\setminus c(S)$ and respecting all forbidden lists on $U$;
\item a partition of $V(H)\setminus (S\cup U)$ into stable pieces
\[
 Q_1,\dots,Q_R;
\]
\item the lightness condition
\[
 w(Q_j)<k\qquad (1\le j\le R).
\]
\end{enumerate}
We call the decomposition \emph{exact} if
\[
 R+q\le 2k-|S|+1.
\]
\end{definition}

\begin{lemma}[Recoloring lemma at palette size $3k$]
Let $T=(S,c,F)$ be a good $C$-template with $|C|\ge 3k$.  Suppose there is an admissible light decomposition with parameters $q,R$ satisfying
\[
 R+q\le 2k-|S|+1.
\]
Then the template can be extended to a full respecting $C$-coloring.
\end{lemma}

\begin{proof}
The set $U$ is already colored using $q$ colors outside $c(S)$.  Let $C'$ be the colors not used on $S$ and not used in the coloring of $U$.  Since $c$ is proper, we have $|c(S)|\le |S|$, and therefore
\[
 |C'|\ge |C|-|c(S)|-q\ge 3k-|S|-q.
\]
For each leftover piece $Q_j$, define its available list
\[
 L(Q_j)=C'\setminus \bigcup_{v\in Q_j}F(v).
\]
Since $w(Q_j)<k$, we have
\[
 |L(Q_j)|\ge |C'|-w(Q_j)> 3k-|S|-q-k.
\]
As all quantities are integral,
\[
 |L(Q_j)|\ge 2k-|S|-q+1.
\]
The exactness condition gives
\[
 R\le 2k-|S|-q+1.
\]
Thus every leftover piece has a list of at least $R$ available colors.  For every subfamily $\calA\subseteq\{Q_1,\dots,Q_R\}$, we have
\[
 \left|\bigcup_{Q_j\in\calA}L(Q_j)\right|
 \ge \max_{Q_j\in\calA}|L(Q_j)|
 \ge R
 \ge |\calA|,
\]
so Hall's condition holds.  Hence, by Hall's theorem (see, for example, Bondy and Murty~\cite{BondyMurty}), one may choose distinct colors
\[
 \gamma_j\in L(Q_j),\qquad 1\le j\le R.
\]
Since each $Q_j$ is stable, color every vertex in $Q_j$ with $\gamma_j$.  The colors are distinct across pieces, avoid the colors used on $S$ and $U$, and avoid every forbidden list.  Hence we obtain a full $C$-coloring respecting
$T$.
\end{proof}

\begin{theorem}[Reduction to exact decomposition]
\label{thm:reduction-to-decomposition}
Suppose that the following statement holds for every minimal $C$-inextensible graph $H$ with a good template $T=(S,c,F)$ and $|C|\ge 3k$:
if
\[
 \chi(H)\le |C|-2k+2,
\]
then $H$ admits an exact admissible light decomposition.  Then
\[
 g(k,m)\le \max(m+2k-2,3k+1)
\]
for all $k,m$.
\end{theorem}

\begin{proof}
Let $G$ satisfy
\[
 \chi(G)\ge \max(m+2k-2,3k+1).
\]
Set $|C|=\chi(G)-1$.  Then $|C|\ge 3k$, and by the empty-template lemma $G$ is $C$-inextensible.  Let $H\subseteq G$ be an induced subgraph minimal with respect to $C$-inextensibility.  By Lemma~\ref{lem:good-template}, $H$ has a
good inextensibility witness.  If
\[
 \chi(H)\le |C|-2k+2,
\]
then the hypothesis gives an exact admissible light decomposition of $H$, and the recoloring lemma produces a respecting $C$-coloring of $H$, a contradiction.  Therefore
\[
 \chi(H)\ge |C|-2k+3=\chi(G)-2k+2.
\]
Since $\chi(G)\ge m+2k-2$, this gives
\[
 \chi(H)\ge m.
\]
By Theorem~\ref{thm:nguyen-connectivity}, the graph $H$ is $(k+1)$-connected.  Thus $H$ is a $(k+1)$-connected subgraph of chromatic number at least $m$, as required.
\end{proof}

Although Theorem~\ref{thm:reduction-to-decomposition} is not invoked in the proof of our main theorem, it isolates the recoloring budget $2k-|S|+1$ that motivates the Hall analysis in Section~4.  In the proof of Theorem~\ref{thm:exact-nguyen-chromatic} we work directly with Nguyen's residual pieces rather than first packaging them as an exact admissible light decomposition.

\section{Residual Hall obstructions}

\noindent We now carry out the final step of the argument inside Nguyen's reduction. After the preliminary recoloring stages, all uncolored vertices lie in a family of stable residual pieces, each equipped with a list of colors that remain legal for that piece.  Nguyen's original proof finishes this stage by a pointwise numerical estimate that requires extra slack in $|C|$.  Our contribution is to replace that last numerical step with a Hall-feasibility analysis of the residual lists: we show that every putative Hall obstruction is incompatible with Nguyen's global budget bounds, thereby allowing Hall's theorem to complete the coloring at the exact threshold.

\medskip\noindent We now record the part of Nguyen's proof of \cite[Lemma~4.1]{Nguyen} that will be used in the Hall analysis.

\begin{lemma}[Reduction package from Nguyen~\cite{Nguyen}]
\label{lem:nguyen-reduction-package}
Let $H$ be a minimal $C$-inextensible graph with a good template $T=(S,c,F)$ and fix a partition
\[
 V(H)=S_1\cup\cdots\cup S_\chi
\]
of $V(H)$ into stable sets.  For each $i\in[\chi]$, set
\[
 P_i:=S_i\setminus S,
 \qquad
 p_i:=\left\lfloor\frac{w(P_i)}{k}\right\rfloor,
 \qquad
 p:=\sum_{i=1}^{\chi}p_i,
 \qquad
 t:=2k-|S|,
 \qquad
 t':=t-p.
\]
Then Nguyen's Lemma~4.5 and the subsequent reduction in his proof of Lemma~4.1 produce the following objects.
\begin{enumerate}[label=(\roman*)]
\item Integers $q_i,t_i$ with
\[
 0\le q_i\le p_i
 \qquad\text{and}\qquad
 p_i=q_i+t_i
 \qquad (i\in[\chi]).
\]
\item Subsets $P_i'\subseteq P_i$ and partitions
\[
 X_i:=P_i\setminus P_i'
 =
 P_{i1}\cup\cdots\cup P_{i,t_i+1}
 \qquad (i\in[\chi]),
\]
where each $P_{ij}$ is stable and satisfies
\[
 w(P_{ij})<k.
\]
\item A proper coloring $c_1$ of
\[
 S^1:=S\cup \bigcup_{i=1}^{\chi}P_i'
\]
respecting the template, such that at most $q_i$ colors are used on $P_i'$.
\item Quantities
\[
 x_i:=w(X_i)-t_i k
 \qquad (i\in[\chi])
\]
satisfying
\[
 0\le x_i<k
 \qquad (i\in[\chi])
\]
and
\[
 \sum_{i=1}^{\chi}x_i<kt'.
\]
(This inequality follows from the cost bound $\operatorname{cost}_k(T)<2k^2$ together with the definitions of $p$, $t$, $t'$, and $x_i$; see Nguyen~\cite{Nguyen}, proof of Lemma~4.1.)

\item Index sets
\[
 I_0:=\{i\in[\chi]:t_i=0\},
 \qquad
 I_2\subseteq[\chi]\setminus I_0,
\]
where $I_2$ is produced by Nguyen's greedy selection procedure (see~\cite{Nguyen}, proof of Lemma~4.1): indices are added to $I_2$ one at a time as long as $\bigcup_{i\in I_2}X_i$ can be colored, extending $c_1$, with at most $\sum_{i\in I_2}t_i$ additional colors while respecting the template, and the procedure terminates when no further index can be added without violating this budget.  We then set
\[
 I_1:=[\chi]\setminus (I_0\cup I_2),
 \qquad
 I:=I_0\cup I_1.
\]

\item An extension $c_2$ of $c_1$ to
\[
 S^2:=S^1\cup\bigcup_{i\in I_2}X_i.
\]
\item For each $i\in I_1$, a minimum-weight part $Y_i$ among $P_{i1},\dots,P_{i,t_i+1}$, and an extension $c_3$ of $c_2$ to
\[
 S^3:=S^2\cup\bigcup_{i\in I_1}(X_i\setminus Y_i)
\]
that uses at most $t_i$ colors on $X_i\setminus Y_i$.  For $i\in I_0$, we set
\[
 Y_i:=X_i.
\]
\end{enumerate}
If we write
\[
 y_i:=w(Y_i)
 \qquad (i\in I),
\]
then each $Y_i\subseteq S_i$ is stable, each $y_i<k$, and the only uncolored vertices after the coloring $c_3$ are the residual pieces $Y_i$ for $i\in I$. Their residual lists are
\[
 L_i:=C\setminus\left(c_3(S^3\setminus S_i)\cup \bigcup_{v\in Y_i}F(v)\right)
 \qquad (i\in I),
\]
where $c_3(S^3\setminus S_i)$ denotes the set of colors used by $c_3$ on vertices of $S^3\setminus S_i$.  Nguyen's calculation also gives
\[
 |c_3(S^3\setminus S_i)|\le 2k-t'-|S\cap S_i|-p_i
 \qquad (i\in I).
\]
Finally, we set
\[
 s_1:=\sum_{i\in I_1}t_i.
\]
For $i\in I_0$, one has $x_i=y_i$.
\end{lemma}

\begin{lemma}[Nguyen's Claim~4.8, restated]
\label{lem:nguyen-claim48}
In the notation of Lemma~\ref{lem:nguyen-reduction-package}, we have
\[
 x_i\ge t_i\bigl(|C|-3k+t'+s_1+|S\cap S_i|+q_i\bigr)
 \qquad (i\in I).
\]
\end{lemma}

\begin{lemma}[Nguyen's Claim~4.10, restated]
\label{lem:nguyen-claim410}
In the notation of Lemma~\ref{lem:nguyen-reduction-package}, assume furthermore that $\chi(H)\le|C|-2k+2$.  Then there exists $J\subseteq I$ with $I_1\subseteq J$, $|J|\le|C|-2k+1$, and $|L_i|\ge|I|$ for every
$i\in I\setminus J$.
\end{lemma}

\begin{definition}[Hall-feasible family]
Let $\calQ$ be a family of stable pieces, each equipped with an available-color list $L(Q)$.  We say that $\calQ$ is \emph{Hall-feasible} if for every subfamily $\calA\subseteq\calQ$,
\[
 \left|\bigcup_{Q\in\calA}L(Q)\right|\ge |\calA|.
\]
A family that is not Hall-feasible will be called \emph{Hall-bad}.
\end{definition}

\subsection*{Standing assumptions for the Hall analysis}

Throughout the rest of this section, we work under the following standing assumptions.  Let \(H\) be a minimal \(C\)-inextensible graph with a good template \(T=(S,c,F)\), and assume
\[
        |C|\ge 3k-1.
\]
We are in the contradiction regime of Nguyen's chromatic lemma, namely
\[
        \chi(H)\le |C|-2k+2.
\]
Write
\[
        |C|=3k-1+d
\]
with \(d\ge 0\).  We fix a partition
\[
        V(H)=S_1\cup\cdots\cup S_\chi
\]
into stable sets and apply Lemma~\ref{lem:nguyen-reduction-package}.  Thus all objects
\[
p_i,p,t,t',q_i,t_i,X_i,x_i,I_0,I_1,I_2,I,Y_i,L_i,s_1
\]
are those produced by that reduction package.  In particular,
\[
        0\le x_i<k,\qquad \sum_i x_i<kt',
\]
and Nguyen's Claims~4.8 and~4.10 are available in this setting.

\noindent We begin by excluding large Hall obstructions.  In this range, \ref{lem:nguyen-claim410} already forces one residual class to have a list large enough to contradict Hall-badness immediately.

\begin{theorem}[No large Hall obstruction]
\label{thm:no-large-hall-obstruction}
Assume $|C|=3k-1+d$ with $d\ge 0$.  There is no Hall-bad family $\calA\subseteq I$ with $|\calA|\ge k+d+1$.
\end{theorem}

\begin{proof}
We use Lemma~\ref{lem:nguyen-claim410}, which gives a set $J\subseteq I$ with $I_1\subseteq J$, $|J|\le|C|-2k+1$, and $|L_i|\ge|I|$ for every $i\in I\setminus J$.  Since $|C|=3k-1+d$, we have
\[
 |C|-2k+1=k+d,
\]
and therefore $|J|\le k+d$.

\medskip\noindent Suppose, toward a contradiction, that $\calA\subseteq I$ is Hall-bad with $|\calA|\ge k+d+1$.  Since $|J|\le k+d$, the family $\calA$ is not contained in $J$, so there exists
\[
 i\in\calA\setminus J.
\]
For this index $i$, Lemma~\ref{lem:nguyen-claim410} gives $|L_i|\ge|I|$. Because $\calA\subseteq I$, we also have $|\calA|\le|I|$.  Therefore
\[
 \left|\bigcup_{h\in\calA}L_h\right|\ge|L_i|\ge|I|\ge|\calA|,
\]
contradicting the assumption that $\calA$ is Hall-bad.
\end{proof}

The remaining difficulty lies in the middle range, where a Hall-bad family is too small for the large-obstruction argument to apply but may still contain indices from $I_1$.  The following numerical lemma is the main new technical ingredient that controls this regime.

\begin{lemma}[numerical lemma]
\label{lem:parametric-I1-numerical}
Let $k,d,r,t'$ be integers with $d\ge 0$, $3\le r\le k+d$, and $1\le t'\le r-d-2$.  Set $A:=k+d-r+t'$ and $M:=r-d-1-t'$.  Let $b\ge 1$, $a+b=r$, and $1\le t_j\le M$ for $1\le j\le b$.  Put $T:=\sum_{j=1}^b t_j$. Then
\[
 aA+\sum_{j=1}^b\max\bigl\{t_j(t'+T+d-1),\,A+t_j(t_j-M)\bigr\}\ge kt'.
\]
\end{lemma}

\begin{proof}
Since $\sum_j\max\{U_j,V_j\}\ge\max\{\sum_j U_j,\sum_j V_j\}$, it suffices to show that at least one of
\[
 D_1:=aA+T(t'+T+d-1)-kt', \qquad
 D_2:=aA+\sum_{j=1}^b\bigl(A+t_j(t_j-M)\bigr)-kt'
\]
is nonnegative.  Set $n:=k+d-r\ge 0$, so $A=n+t'$ and $r=M+d+t'+1$. Using $a=r-b$,
\begin{align}
 D_1 &= n(M+d+1-b)+T(T+d-1)+t'(T+d-b), \tag{1}\label{eq:D1-parametric}\\
 D_2 &= n(M+d+1)+dt'+\sum_{j=1}^b t_j(t_j-M). \tag{2}\label{eq:D2-parametric}
\end{align}
\medskip\noindent For the Case $b\le M+d+1$: The first term of \eqref{eq:D1-parametric} is nonnegative.  Since $T\ge b$, we have $T+d-b\ge d\ge 0$, and $T(T+d-1)\ge 0$ since $T\ge 1$.  Hence $D_1\ge 0$.

\medskip\noindent For the Case $b>M+d+1$:  Assume $D_1<0$.  Since $T\ge b$ and $t'\ge 1$, the term $t'(T+d-b)\ge 0$ may be discarded from \eqref{eq:D1-parametric}, giving $n(b-M-d-1)>T(T+d-1)$.  Set $B:=b-M-d-1>0$, so $n>T(T+d-1)/B$ and hence
\[
 n(M+d+1)>\frac{T(T+d-1)(M+d+1)}{B}.
\]
We claim $(T+d-1)(M+d+1)\ge(M-1)B$.  Indeed, $T\ge b=M+d+1+B$ gives $T+d-1\ge M+2d+B$, so
\[
 (T+d-1)(M+d+1)\ge(M+2d+B)(M+d+1)=(M+2d)(M+d+1)+B(M+d+1)\ge B(M-1),
\]
where the last step uses $(M+2d)(M+d+1)\ge 0$ and $(M+d+1)-(M-1)=d+2\ge 0$.  Therefore
\begin{equation}
 n(M+d+1)>T(M-1). \tag{3}\label{eq:n-bound-parametric}
\end{equation}
Since $t_j\ge 1$, we have $t_j(t_j-M)=t_j^2-Mt_j\ge t_j-Mt_j=-t_j(M-1)$, and summing over $j$ gives $\sum_j t_j(t_j-M)\ge-T(M-1)$. Substituting \eqref{eq:n-bound-parametric} into \eqref{eq:D2-parametric},
\[
 D_2>T(M-1)+dt'-T(M-1)=dt'\ge 0.
\]
Thus one of $D_1,D_2$ is nonnegative in every case.
\end{proof}

We now use the numerical lemma to rule out middle-range Hall obstructions that meet $I_1$.  This is the heart of the parametric strengthening, since it is precisely here that Nguyen's original proof required additional slack.

\begin{theorem}[Exclusion of $I_1$-indices]
\label{thm:parametric-I1-exclusion}
Assume $|C|=3k-1+d$ with $d\ge 0$.  Let $\calA\subseteq I$ be a Hall-bad family with $3\le|\calA|\le k+d$.  Then $\calA\cap I_1=\varnothing$.
\end{theorem}

\begin{proof}
Set $r:=|\calA|$.  Since $\calA$ is Hall-bad, $|\bigcup_{i\in\calA}L_i|<r$, so in particular $|L_i|\le r-1$ for every $i\in\calA$.  Recalling that
\[
 L_i=C\setminus\!\left(c_3(S^3\setminus S_i)\cup\bigcup_{v\in Y_i}F(v)\right),
\]
Nguyen's bound $|c_3(S^3\setminus S_i)|\le 2k-t'-|S\cap S_i|-p_i$ together with $|C|=3k-1+d$ gives
\[
 |L_i|\ge(3k-1+d)-(2k-t'-|S\cap S_i|-p_i)-y_i
        =k-1+d+t'+|S\cap S_i|+p_i-y_i.
\]
Combining with $|L_i|\le r-1$ yields
\begin{equation}
\label{eq:param-yi-lower}
 y_i\ge k+d-r+t'+|S\cap S_i|+p_i \qquad(i\in\calA).
\end{equation}

\bigskip\noindent Suppose for contradiction that $\calA\cap I_1\neq\varnothing$.  Write $B:=\calA\cap I_1$ and $A_0:=\calA\cap I_0$, with $b:=|B|\ge 1$ and $a:=|A_0|$, so $a+b=r$.

\medskip\noindent\textit{Lower bounds on $x_i$:}
For $i\in A_0$: since $t_i=0$ we have $X_i=Y_i$ and $x_i=y_i$, so \eqref{eq:param-yi-lower} (dropping the nonnegative terms $|S\cap S_i|$ and $p_i$) gives
\begin{equation}
\label{eq:param-I0-lower}
 x_i\ge k+d-r+t' =: A.
\end{equation}
For $i\in B$: since $Y_i$ is a minimum-weight piece among the $t_i+1$ parts of $X_i$, we have $(t_i+1)y_i\le w(X_i)=x_i+t_ik$, hence $x_i\ge(t_i+1)y_i-t_ik$.  Substituting \eqref{eq:param-yi-lower} and using $p_i=q_i+t_i\ge t_i$ and $|S\cap S_i|\ge 0$, we obtain
\begin{equation}
\label{eq:param-I1-first-lower}
 x_i\ge k+d-r+t'+t_i(t'-r+d+t_i+1).
\end{equation}
Setting $T:=\sum_{i\in B}t_i$ and applying Lemma~\ref{lem:nguyen-claim48} with $|C|=3k-1+d$ and $T\le s_1$,
\begin{equation}
\label{eq:param-I1-second-lower}
 x_i\ge t_i(t'+T+d-1) \qquad(i\in B).
\end{equation}

\medskip\noindent\textit{Range of $t':$}
Since $0\le\sum_{i\in I}x_i<kt'$, we have
\begin{equation}
\label{eq:param-tprime-positive}
 t'\ge 1.
\end{equation}
For $i\in B$, the lightness condition $y_i<k$ gives $y_i\le k-1$. Combining with \eqref{eq:param-yi-lower} and $p_i=q_i+t_i\ge 1$ (since $i\in I_1$) yields $r-1\ge d+t'+|S\cap S_i|+p_i\ge d+t'+1$, so
\begin{equation}
\label{eq:param-tprime-range}
 t'\le r-d-2.
\end{equation}
If $r\le d+2$ then \eqref{eq:param-tprime-range} contradicts \eqref{eq:param-tprime-positive}, so $r\ge d+3$.  Moreover $t_i\le p_i\le r-d-1-t'$.

\medskip\noindent
Set $A:=k+d-r+t'$ and $M:=r-d-1-t'$.  Then \eqref{eq:param-I0-lower} reads $x_i\ge A$ for $i\in A_0$, and since $t'-r+d+t_i+1=t_i-M$, \eqref{eq:param-I1-first-lower} reads $x_i\ge A+t_i(t_i-M)$ for $i\in B$. The constraints \eqref{eq:param-tprime-positive} and \eqref{eq:param-tprime-range} verify the hypothesis $1\le t'\le r-d-2$ of Lemma~\ref{lem:parametric-I1-numerical}.  Combining the lower bounds above with \eqref{eq:param-I1-second-lower} and applying that lemma gives $\sum_{i\in\calA}x_i\ge kt'$.  But $\sum_{i\in I}x_i<kt'$ and all $x_i\ge 0$, so $\sum_{i\in\calA}x_i\le\sum_{i\in I}x_i<kt'$, a contradiction.  Hence $\calA\cap I_1=\varnothing$.
\end{proof}

Once the $I_1$-indices have been excluded, the remaining middle-range obstruction is an $I_0$-only family.  In that case the argument becomes simpler, because the residual weight parameter agrees with the surplus parameter.

\begin{theorem}[No middle-size $I_0$-only Hall obstruction]
\label{thm:no-middle-I0-parametric}
Assume $|C|=3k-1+d$ with $d\ge 0$.  There is no Hall-bad family $\calA\subseteq I_0$ with $3\le|\calA|\le k+d$.
\end{theorem}

\begin{proof}
Suppose such $\calA$ exists and set $r:=|\calA|$.  Since $\sum x_i<kt'$ and all $x_i\ge 0$, we have $t'\ge 1$.  Arguing as in
Theorem~\ref{thm:parametric-I1-exclusion}, Hall-badness and $t_i=0$ for $i\in I_0$ give $x_i=y_i\ge k+d-r+t'$ for all $i\in\calA$, hence
\[
 \sum_{i\in\calA}x_i\ge r(k+d-r+t').
\]
It remains to show $r(k+d-r+t')\ge kt'$, equivalently $r(k+d-r)+(r-k)t'\ge 0$.  If $r\ge k$ this is immediate since $r\le k+d$. If $r<k$, then $y_i\le k-1$ and \eqref{eq:param-yi-lower} force $t'\le r-d-1$; if $r\le d+1$ this contradicts $t'\ge 1$, so $r\ge d+2$ and
\[
 r(k+d-r)-(k-r)t'\ge r(k+d-r)-(k-r)(r-d-1)=k(d+1)-r>0.
\]
Hence $\sum_{i\in\calA}x_i\ge kt'$, contradicting $\sum_{i\in I}x_i<kt'$.
\end{proof}

With the residual Hall theorem in hand, Nguyen's chromatic lemma can be completed at the exact threshold.  The only change from Nguyen's proof is that the final slack-based argument is replaced by Hall's theorem on the residual family.

\begin{theorem}[Residual Hall feasibility above the exact threshold]
\label{thm:residual-hall-feasible-parametric}
Assume $|C|=3k-1+d$ with $d\ge 1$.  Then the residual family $(Y_i)_{i\in I}$ is Hall-feasible. 
That is, for every $\calA\subseteq I$,
$$\left|\bigcup_{i\in \calA}L_i\right| \geq |\calA|.$$
\end{theorem}

\begin{proof}
Suppose $\calA\subseteq I$ is Hall-bad and set $r:=|\calA|$.  Since $\sum x_i<kt'$ and all $x_i\ge 0$, we have $t'\ge 1$.  Nguyen's bound gives $|L_i|\ge d+t'+|S\cap S_i|+p_i\ge d+t'\ge 2$ for all $i\in I$, ruling out $r=1$.  For $r=2$, Hall-badness forces $|L_i\cup L_j|=1$, contradicting $|L_i|\ge 2$.  For $3\le r\le k+d$, Theorem~\ref{thm:parametric-I1-exclusion} gives $\calA\subseteq I_0$, contradicting Theorem~\ref{thm:no-middle-I0-parametric}.  For $r\ge k+d+1$,
Theorem~\ref{thm:no-large-hall-obstruction} gives a contradiction.  Hence no Hall-bad family exists.
\end{proof}

The final step is now formal.  We combine the exact chromatic lemma for minimal $C$-inextensible graphs with Nguyen's connectivity theorem to obtain the desired bound for $g(k,m)$.

\begin{theorem}[Exact version of Nguyen's chromatic lemma]
\label{thm:exact-nguyen-chromatic}
Let $H$ be a minimal $C$-inextensible graph with a good template $T=(S,c,F)$. If $|C|\ge 3k$, then $\chi(H)\ge |C|-2k+3$.
\end{theorem}

\begin{proof}
\noindent Suppose not, and let $\chi:=\chi(H)\le|C|-2k+2$.  Write $|C|=3k-1+d$ with $d\ge 1$.  Choose a partition $V(H)=S_1\cup\cdots\cup S_\chi$ of $V(H)$ into stable sets, set $P_i:=S_i\setminus S$, and define
\[
 p_i=\left\lfloor\frac{w(P_i)}{k}\right\rfloor,
 \qquad
 p=\sum_i p_i,
 \qquad
 t=2k-|S|,
 \qquad
 t'=t-p.
\]
\medskip\noindent Apply Lemma~\ref{lem:nguyen-reduction-package}.  This produces the sets $X_i$, the parameters $q_i,t_i,x_i$, the index sets $I_0$, $I_1$, $I_2$, $I=I_0\cup I_1$, the intermediate colorings $c_1$, $c_2$, $c_3$ on $S^1\subseteq S^2\subseteq S^3$, and the residual pieces $(Y_i)_{i\in I}$ with residual lists $(L_i)_{i\in I}$.  By construction, the only uncolored vertices after $c_3$ are the stable sets $Y_i$ for $i\in I$.

\medskip\noindent At the final residual-coloring step, instead of Nguyen's final numerical argument, we apply Theorem~\ref{thm:residual-hall-feasible-parametric}, which gives that the family $(Y_i)_{i\in I}$ is Hall-feasible:
\[
 \left|\bigcup_{i\in\calA}L_i\right|\ge|\calA| \qquad (\calA\subseteq I).
\]
\medskip\noindent By Hall's theorem, there exist pairwise distinct colors $\gamma_i\in L_i$ for each $i\in I$.  Color every vertex of $Y_i$ by $\gamma_i$.  This extends $c_3$ to a proper $C$-coloring of $H$: each $Y_i\subseteq S_i$ is stable and the colors $\gamma_i$ are distinct for distinct $i$, so there is no conflict among the residual pieces themselves; since $\gamma_i\notin c_3(S^3\setminus S_i)$, no vertex of $Y_i$ conflicts with any already-colored vertex outside $S_i$; there are no edges from $Y_i$ to already-colored vertices inside $S_i$ because $S_i$ is stable; and $\gamma_i\notin F(v)$ for $v\in Y_i$ because $\gamma_i\in L_i$ by definition.

\medskip\noindent Therefore the coloring extends to a proper $C$-coloring of all of $H$ respecting the template $T=(S,c,F)$, contradicting $C$-inextensibility of $H$. Hence $\chi(H)\ge|C|-2k+3$.
\end{proof}

\section{Proof of Theorem \ref{thm:main}}
\noindent The remaining step is now formal: we combine the exact chromatic lemma for minimal $C$-inextensible graphs with Nguyen's connectivity theorem to obtain the desired bound for $g(k,m)$.

\begin{proof}[Proof of Theorem \ref{thm:main}]
Let $G$ satisfy
\[
 \chi(G)\ge \max(m+2k-2,3k+1).
\]
Set
\[
 |C|=\chi(G)-1.
\]
Then $|C|\ge 3k$, and by Lemma~\ref{lem:good-template} together with the empty template lemma, every induced subgraph minimal with respect to being $C$-inextensible admits a good inextensibility witness.  Let $H\subseteq G$ be such a minimal induced $C$-inextensible subgraph.  By Theorem~\ref{thm:exact-nguyen-chromatic},
\[
 \chi(H)\ge |C|-2k+3=\chi(G)-2k+2.
\]
Since $\chi(G)\ge m+2k-2$, we obtain
\[
 \chi(H)\ge m.
\]
By Theorem~\ref{thm:nguyen-connectivity}, the graph $H$ is $(k+1)$-connected.  Thus $H$ is a $(k+1)$-connected subgraph of chromatic number at least $m$, as required.
\end{proof}

\begin{remark}
The bound $|C|\ge 3k$ in Theorem~\ref{thm:exact-nguyen-chromatic} is sharp for the present method.  The Hall-feasibility argument breaks down at $|C|=3k-1$, i.e.\ $d=0$, already at the smallest possible obstruction size: when $d=0$ the list size lower bound gives only $|L_i|\ge t'\ge 1$, and the inequality $|L_i|\ge d+t'\ge 2$ used to rule out Hall-bad pairs in the proof of Theorem~\ref{thm:residual-hall-feasible-parametric} requires $d\ge 1$ strictly.  Thus any improvement of the upper bound on $g(k,k+1)$ below $3k+1$ would require either a fundamentally different approach to the residual coloring step, or a refinement of the reduction package that produces residual lists with additional structure not captured by the present framework.  We also note that the best known lower bound $g(k,k+1)\ge 2k$, due to~\cite{AKSST}, leaves open the question of whether $g(k,k+1)=3k+1$ exactly; closing this gap would require a new extremal construction.
\end{remark}

\bibliography{bibliography}
\bibliographystyle{alpha}
\end{document}